\documentclass[a4paper,reqno]{amsart}

\pdfoutput=1

\usepackage{amssymb,amsmath,amsthm,mathrsfs,graphicx,subcaption,geometry}
\usepackage[all]{xy} 
\usepackage{pstricks,pstricks-add} 
\usepackage{epsfig}
\usepackage{hyperref}
\newtheorem{mthm}{Main Theorem}

\newtheorem{thm}{Theorem}[section]
 \newtheorem{cor}[thm]{Corollary}
 \newtheorem{lem}[thm]{Lemma}
 \newtheorem{prop}[thm]{Proposition}
 
 \theoremstyle{definition}
 \newtheorem{defn}[thm]{Definition}
 \theoremstyle{remark}
 \newtheorem{rem}[thm]{Remark}
 
 \newtheorem{ex}[thm]{Example}
 \newcommand{\C}{\mathbb{C} }
 \newcommand{\CC}{\widehat{\C} }
 \newcommand{\R}{\mathbb{R} }
 
 \newcommand{\Z}{\mathbb{Z} }

 \newcommand{\mate}{\perp \! \! \! \perp}

\begin{document}
 
 \title[Obstructions for cubic Thurston maps with two fixed critical points]{Matings of Cubic polynomials with a fixed critical point, Part I: Thurston Obstructions}
\author{Thomas Sharland}
\address{Department of Mathematics, University of Rhode Island, Kingston, Rhode Island, USA}
\email{tsharland@uri.edu}
\date{\today}
 
 \begin{abstract}
 We prove that if $F$ is a degree $3$ Thurston map with two fixed critical points, then any irreducible obstruction for $F$ contains a Levy cycle. As a corollary, it will be shown that if $f$ and $g$ are two postcritically finite cubic polynomials each having a fixed critical point, then any obstruction to the mating $f \mate g$ contains a Levy cycle. We end with an appendix to show examples of the obstructions described in the paper.
\end{abstract}
\maketitle
 
 \section{Introduction} 
 The study of branched coverings of the sphere is, by virtue of Thurston's theorem, an important area in the field of holomorphic dynamics. Thurston's Theorem gives a combinatorial criterion for checking when a postcritically finite branched covering is \emph{equivalent} to a rational map on the Riemann sphere. This condition requires the checking of collections of curves (known as multicurves) in the complement of the postcritical set to see if they form an \emph{obstruction}. However, in general it is very hard to know what these obstructions may look like, and so often one must restrict to particular families of postcritically finite branched coverings to investigate what kind of obstructions may occur.
 
 In the quadratic (or more generally, bicritical) case, it was shown through work of Rees, Shishikura and Tan \cite{Tan:quadmat} that any obstruction contained a \emph{Levy cycle} (see Definition~\ref{d:Levy}). However, Shishikura and Tan \cite{ShishTan} showed that there exist obstructions for postcritically finite cubic branched coverings which were not Levy cycles, suggesting that the degree $3$ (and higher) case is more complicated. On the other hand, Tan, whilst studying the family of cubic Newton maps, \cite{Tan:Newton} showed that any obstruction for a postcritically finite cubic branched covering with three fixed critical points contains a Levy cycle. This present paper strengthens this latter result, showing the result holds in the case that the branched covering has only two fixed critical points.
 
 One particular way of constructing a branched covering of the sphere is to form the \emph{mating} of two monic polynomials of the same degree. Informally, one may think of the mating operation as gluing together the dynamics of the two polynomials to construct a branched covering. In the postcritically finite case, one may then check, via Thurston's theorem, if this branched covering is equivalent to a rational map on the Riemann sphere. The results cited for general postcritically finite branched coverings in the previous paragraph apply in particular for matings. Furthermore, building on Tan's work in \cite{Tan:Newton}, Aspenberg and Roesch \cite{AspRoe} studied matings of polynomials (not necessarily postcritically finite) in the family of cubic Newton maps. However, these maps still had three fixed critical points, forcing one of the polynomials in the mating to be the ``double basilica''\footnote{In this context, the \emph{double basilica} is the unique, up to conjugacy, cubic polynomial with two fixed critical points}. The present paper in concerned with cubic postcritically finite branched coverings with two fixed critical points. The sequel to this paper will concentrate on the matings of postcritically finite monic cubic polynomials each possessing a fixed critical point. The branched coverings formed by the mating of two such polynomials are then cubic Thurston maps with two fixed critical points.

  The author is grateful to the anonymous referee for suggestions and remarks which helped improve the article.

 \subsection{Thurston's Theorem}
 
 Let $F \colon \Sigma \to \Sigma$ be an orientation-preserving branched self-covering of a topological $2$-sphere. We denote by $\Omega_F$ the critical set of $F$ and define
\[
 P_F = \bigcup_{n > 0} F^{n}(\Omega_F)
\]
to be the postcritical set of $F$. We say that $F$ is postcritically finite if $|P_F| < \infty$. Following \cite{BonkMeyer}, we call $F \colon \Sigma \to \Sigma$ a \emph{Thurston map} if it is a postcritically finite orientation-preserving branched self-covering of a topological $2$-sphere.
 
 \begin{defn}\label{d:Thurst}
Let $F \colon \Sigma \to \Sigma$ and $\widehat{F} \colon \widehat{\Sigma} \to \widehat{\Sigma}$ be Thurston maps. An \emph{equivalence} is given by a pair of orientation-preserving homeomorphisms $(\Phi,\Psi)$ from $\Sigma$ to $\widehat{\Sigma}$ such that 
    \begin{itemize}
     \item{$\Phi |_{P_{F}} = \Psi |_{P_{F}}$}
     \item{The following diagram commutes:
        \[
             \xymatrix{       (\Sigma,P_F) \ar[r]^{\Psi} \ar[d]_{F}    & (\widehat{\Sigma},P_{\widehat{F}}) \ar[d]^{\widehat{F}}
            \\ 
                (\Sigma,P_F) \ar[r]_{\Phi}                       & (\widehat{\Sigma},P_{\widehat{F}}) }
        \]}
    \item{$\Phi$ and $\Psi$ are isotopic via a family of homeomorphisms $t \mapsto \Phi_{t}$ which is constant on $P_F$.}
        \end{itemize}
\end{defn}

If there exists an equivalence as above, we say that $F$ and $\widehat{F}$ are equivalent. Note that in particular, a postcritically finite rational map $R \colon \CC \to \CC$ on the Riemann sphere is a Thurston map. Hence it is natural to ask when a general Thurston map is equivalent to a rational map.

\begin{defn}
Let $F$ be a Thurston map. A \emph{multicurve} is a collection $\Gamma = \{ \gamma_1, \ldots,\gamma_n \}$ of simple, closed,  non-peripheral\footnote{A simple closed curve $\gamma$ is non-peripheral if each component of $\Sigma \setminus \gamma$ contains at least two elements of $P_F$.} curves such that each $\gamma_i \in \Gamma$ is disjoint from each other $\gamma_j$ and the $\gamma_i$ are pairwise non-homotopic relative to $P_F$. A multicurve is called $F$-\emph{stable} if for all $\gamma_{i} \in \Gamma$, all the non-peripheral components of $F^{-1}(\gamma_{i})$ are homotopic relative to $P_F$ to elements of $\Gamma$. If $\Gamma$ is a multicurve (not necessarily $F$-stable), we define the non-negative matrix $F_\Gamma = (f_{ij})_{n \times n}$ as follows. Let $\gamma_{i,j,\alpha}$ be the components  of $F^{-1}(\gamma_{j})$ which are homotopic to $\gamma_{i} \in \Gamma$ in $\Sigma \setminus P_{F}$. Now define
\[
	F_{\Gamma}(\gamma_{j}) = \sum_{i,\alpha} \frac{1}{\deg F |_{\gamma_{i,j,\alpha}} \colon \gamma_{i,j,\alpha} \to \gamma_{j}} \gamma_{i}.
\] 
where $\deg$ denotes the degree of the map. By standard results on non-negative matrices (see \cite{Gantmacher}), this matrix $(f_{ij})$ will have a leading non-negative eigenvalue $\lambda$. We write $\lambda(\Gamma)$ for the leading eigenvalue associated to the multicurve $\Gamma$.
\end{defn}

The importance of the above is due to the following rigidity theorem. A proof can be found in \cite{DouadyHubbard:Thurston}.

 \begin{thm}[Thurston]\mbox{}
\begin{enumerate}
  \item{A Thurston map $F \colon \Sigma \to \Sigma$ of degree $d \geq 2$ with hyperbolic orbifold is equivalent to a rational map $R \colon \CC \to \CC$ if and only if there are no $F$-stable multicurves with $\lambda(\Gamma) \geq 1$.}
  \item{Any Thurston equivalence of rational maps $F$ and $\widehat{F}$ with hyperbolic orbifolds is represented by a M\"{o}bius conjugacy.}
\end{enumerate}
\end{thm}

An $F$-stable multicurve with $\lambda(\Gamma) \geq 1$ is called a \emph{(Thurston) obstruction}. The condition that $F$ has a hyperbolic orbifold is a purely combinatorial one; this condition can be checked by inspecting the dynamics on the union of the critical and postcritical sets of $F$. We remark that since the maps in this note have two fixed critical points, the respective orbifolds are guaranteed to be hyperbolic. For further details, see \cite{DouadyHubbard:Thurston}. 

\begin{defn}\label{d:Levy}
 A multicurve $\Gamma = \{ \gamma_{1}, \ldots , \gamma_{n} \}$ is a Levy cycle if for each $i =1,\ldots,n$, the curve $\gamma_{i-1}$ (or $\gamma_{n}$ if $i = 1$) is homotopic to some component $\gamma_{i}'$ of $F^{-1}(\gamma_{i})$ (rel $P_{F}$) and the map $F \colon \gamma_{i}' \to \gamma_{i}$ is a homeomorphism. We say $\Gamma$ is a degenerate Levy cycle if the connected components of $\Sigma \setminus \bigcup_{i=1}^n \gamma_i$ are $D_1, \ldots ,D_n, C$, where each $D_i$ is a disk and moreover for each $i$ the preimage $F^{-1}(D_{i+1})$ contains a component $D_{i+1}'$ which is isotopic to $D_i$ relative to $P_F$ and is such that $F \colon D_{i+1}' \to D_{i+1}$ is a homeomorphism. A degenerate Levy cycle is called a removable Levy cycle if in addition, for all $k \geq 1$ and all $i$, the components of $F^{-k}(D_i)$ are disks.
\end{defn}  

By Lemma 2.2 of \cite{Tan:quadmat}, a Levy cycle can be completed to a Thurston obstruction.

 \subsection{Background on branched coverings}
 
 The original definition of Thurston obstruction requires us to search for $F$-stable multicurves which are obstructions. By Proposition~\ref{p:irred}, it is sufficient to consider \emph{irreducible obstructions}. The advantage of restricting attention to irreducible obstructions is that they are simpler than Thurston obstructions, since they can be considered as the fundamental part of the obstruction. 
 
 \begin{defn}
  A multicurve is said to be irreducible if its associated matrix $F_\Gamma$ is irreducible. A multicurve $\Gamma$ is an irreducible obstruction if the matrix $F_\Gamma$ is irreducible and if $\lambda(\Gamma) \geq 1$.
 \end{defn}

 Recall that an $n \times n$ matrix $A_{ij}$ is irreducible if for all $(i,j)$, there exists $k >0$ such that the element $(A^k)_{ij}>0$. In terms of a multicurve $\Gamma = \{ \gamma_1, \ldots, \gamma_n \}$, this means that for each pair $(i,j) \in \{1,\ldots,n\}^2$, there exists an integer $k>0$ and a component $\gamma'$ of $F^{-k}(\gamma_j)$ which is isotopic to $\gamma_i$, and that for each $1 \leq \ell \leq k$, the curve $F^{\ell} (\gamma')$ is isotopic to a curve in $\Gamma$. The following result (see \cite{ShishTan}) ties together the notions of a Thurston obstruction and an irreducible obstruction.
 
 \begin{prop}\label{p:irred}
  If $F$ has a hyperbolic orbifold, then $F$ is not equivalent to a rational map if and only if $F$ has an irreducible obstruction.
 \end{prop}

 We remark that a Levy cycle is an example of an irreducible obstruction. Now let $F \colon \Sigma \to \Sigma$ be a postcritically finite branched covering and $A,B \subseteq \Sigma$. We say that $A$ is \emph{isotopically contained in} $B$ relative to $P_F$ if there exists a homeomorphism $\phi \colon \Sigma \to \Sigma$ which is isotopic to the identity relative to $P_F$ and for which $\phi(A) \subseteq B$. The following two propositions from \cite{ShishTan} will be helpful in our analysis of irreducible obstructions. 
 
 \begin{prop}\label{p:isotopiccontain}
  Let $\Gamma$ be an irreducible obstruction for a postcritically finite branched covering $F$. Then
  \begin{enumerate}
  \item Each connected component of $\Sigma \setminus F^{-1}(\Gamma)$ is exactly a connected component of $F^{-1}(A)$, where $A$ is some connected component of $\Sigma \setminus \Gamma$.
  \item Any connected component of $\Sigma \setminus F^{-1}(\Gamma)$ is isotopically contained in a connected component of $\Sigma \setminus \Gamma$ relative to $P_F$.
  \end{enumerate}
 \end{prop}
 
 \begin{prop}\label{p:critvals}
 $\Gamma$ is a removable Levy cycle for $F$ if and only if $\Gamma$ is an irreducible obstruction for $F$ and there exists a disk component $D$ of $\Sigma \setminus \Gamma$ such that for all $n$, the components of $F^{-n}(D)$ are disks.
\end{prop}

We will also make use of the following classical result (see e.g \cite{MilnorComplex}, Theorem 7.2).

\begin{thm}[Riemann-Hurwitz Theorem]\label{t:RH}
 Let $F \colon Y \to X$ be a degree $d$ branched covering of compact Riemann surfaces. Then
 \[
  \chi(Y) = d \chi(X) - \sum_{P \in Y} (e_P - 1),
 \]
where $\chi(Y)$ and $\chi(X)$ are the Euler characteristics of $Y$ and $X$ respectively, and $e_P$ is the local degree of $F$ at $P$.
\end{thm}

Since there are only finitely many points in $Y$ such that $e_P \neq 1$, the final term in the equation is well-defined. The particular circumstances in which we will need Theorem~\ref{t:RH} are summarised in the following lemma. 
 
 \begin{lem}\label{l:RH}
  Suppose $F$ is a cubic branched covering of the Riemann sphere with simple critical points and $D \subseteq \CC$ is a disk. Then
  \begin{enumerate}
   \item[(a)] If $D$ contains precisely two critical values and $F^{-1}(D)$ contains a non-disk component then $F^{-1}(D)$ consists of an annulus $A$ and a disk $D'$. Furthermore, the boundary curves of $A$ and the curve $\partial D'$ both map by degree $1$ onto $\partial D$. Both the critical points corresponding to the two critical values in $D$ are contained in $A$.
   \item[(b)] If $D$ contains precisely one critical value then $F^{-1}(D)$ consists of two disks. The boundary curve of one of the disks maps by degree $2$ onto $\partial D$ and the boundary curve  of the other disk maps by degree $1$ onto $\partial D$.
   \item[(c)] If $D$ contains precisely no critical values then $F^{-1}(D)$ is the union of three disks. The boundary curve of each preimage disk maps by degree $1$ onto $\partial D$.
  \end{enumerate}
 \end{lem}

 \begin{proof}
  We first remark that if $D$ is a closed disk then $\chi(D) = 1$ and that if $U$ is a closed disk from which $n$ open disks have been removed then $\chi(U) = 1 - n$. Let $Y = F^{-1}(D)$.
  \begin{itemize}
   \item[(a)] By Theorem~\ref{t:RH} we see that $\chi(Y) = 1$. By assumption $Y$ cannot be a disk, and so $Y$ consists of at least two connected components. Furthermore, each boundary curve of $Y$ maps onto $\partial D$ and since $F$ has degree $3$, we see that $Y$ has at most three boundary curves. The only possibility is then that $Y$ is the union of an annulus $A$ and a disk $D'$ and each boundary curve of $Y$ maps homeomorphically onto $D$. Since $D'$ must therefore map homeomorphically onto $D$, we see that it cannot contain any critical points of $F$. Thus both critical points of $F$ that lie in $Y$ are both contained in $A$.
   \item[(b)] By Theorem~\ref{t:RH}, we see that $\chi(Y) = 2$. It follows that $Y$ is the disjoint union of two disks, $\Delta_1$ and $\Delta_2$. Since $\partial D$ is mapped by degree $3$, we see that the boundary curve of one of the disks must map by degree $2$ onto $\partial D$. The boundary curve of the other disk must then map homeomorphically onto $\partial D$.
   \item[(c)] Since $D$ contains no critical values, it follows that $F \colon Y \to D$ is a degree $3$ cover of $D$ and so $Y$ is the union of three disks. Furthermore, the boundaries of the three disks making up $Y$ are precisely the preimages of $\partial D$. Thus each boundary must map onto $\partial D$ homeomorphically.
  \end{itemize}

 \end{proof}

 \begin{rem}
  Note that the assumption that $F^{-1}(D)$ contains a non-disk component in part (a) is necessary. Otherwise, it is possible that if $D$ contains only two critical values, the set $F^{-1}(D)$ could consist of just one disk, which maps by degree $3$ to $D$ (for example, this case is realised if $D$ is one of the hemispheres in a mating).
 \end{rem}
 
By Proposition~\ref{p:critvals}, if an irreducible obstruction is not a removable Levy cycle, then there exists a disk component $D$ such that $F^{-1}(D)$ contains a non-disk component. It is easy to see that by Theorem~\ref{t:RH}, the disk $D$ must contain at least two critical values of $F$.
 
\subsection{Statement of the Main Theorems}
 
 Our main theorem is the following.
 
 \begin{mthm}\label{mthm}
 Let $F$ be a cubic Thurston map on a topological sphere $\Sigma$ with two fixed critical points. Then if $\Gamma$ is an irreducible obstruction for $F$, then $\Gamma$ contains a Levy cycle.
 \end{mthm}

 \begin{rem}
  It is actually sufficient to prove the theorem in the case that the fixed critical points are simple (i.e. have local degree $2$). If both critical points have local degree $3$ then $F$ is equivalent to the map $z \mapsto z^3$. If precisely one of the fixed critical points has local degree $3$ then $F$ is a topological polynomial and then all possible obstructions are (degenerate) Levy cycles \cite{BFH}.
 \end{rem}
 
 The sequel to this paper (\cite{CubicmatingsPt2}) will use this result to analyse the matings of polynomials in $\mathcal{S}_1$, the space of cubic polynomials with a fixed marked critical point. An immediate consequence of Main Theorem \ref{mthm} is the following. 
 
 \begin{mthm}\label{mthm2}
  Let $f$ and $g$ be monic postcritically finite polynomials of degree $3$, each with a fixed critical point. Then any obstruction to the formal mating $f \uplus g$ contains a Levy cycle. 
 \end{mthm}

 \section{Proof of Main Theorem \ref{mthm}}

We now analyse the irreducible obstructions for cubic Thurston maps with two fixed critical points. We will see that, if $\Gamma$ is not a removable Levy cycle, then there are two possible cases, a ``Newton-like case'' and a ``quadratic-like case''. These are summarised in the following proposition.

\begin{prop}\label{p:notremov}
Let $\Gamma$ be an irreducible obstruction for a cubic Thurston map with two fixed critical points, and suppose that $\Gamma$ is not a removable Levy cycle. Then there exists a disk $D$ such that $F^{-1}(D)$ contains a non-disk component. Furthermore, precisely one of the following two cases occurs.
\begin{itemize}
 \item \emph{(Quadratic-like case)}. The disk $D$ contains precisely two critical values, neither of which is a fixed critical point.
 \item \emph{(Newton-like case)}. The disk $D$ contains a fixed critical point.
\end{itemize}
\end{prop}

\begin{proof}
 By Proposition~\ref{p:critvals} there must be a disk component of $F^{-1}(D)$ which has a non-disk component. It follows easily from Lemma~\ref{l:RH} that $D$ must contain (at least) two critical values of $F$. If one of these critical values is also a fixed critical point, we are in the Newton-like case. Otherwise, the disk $D$ must contain precisely two critical values (the two critical values which are not also fixed critical points), which means we are in the quadratic-like case.
\end{proof}

We deal with both of the cases from Proposition~\ref{p:notremov} in turn, showing that in each case any irreducible obstruction will contain a Levy cycle. 
 
 \subsection{Quadratic-like case}
 
 This case is similar to the proof that an irreducible obstruction to a quadratic (or bicritical) postcritically finite branched cover of the sphere is a Levy cycle \cite{Tan:quadmat}. However, there is an extra difficulty in the cubic case. In the quadratic case, the existence of a Levy cycle is ensured by observing that every curve $\gamma$ in an irreducible obstruction has precisely two preimages. This means that the preimages of $\gamma$ map homeomorphically onto $\gamma$, and so any periodic cycle in the obstruction must be a Levy cycle. In the following, we will show that in the cubic case, all curves in an irreducible obstruction have precisely three preimages, and so a similar argument to that outlined above shows that the obstruction must contain a Levy cycle.
 
 By assumption, the quadratic-like case occurs when there is a disk component $D$ containing two critical values of $F$ which are not fixed critical points of $F$, and such that $F^{-1}(D)$ contains a non-disk component. By Lemma~\ref{l:RH}, the preimage of this disk is an annulus (which contains two critical points) and a disk. The annulus has two complementary regions. One of these, $U_1$, maps homeomorphically onto $S^2 \setminus D$, whilst the other, $U_2$, which contains both fixed critical points, maps as a degree $2$ branched covering to $S^2 \setminus D$, see Figure~\ref{f:quadlike}.
 
 \begin{figure}[ht]
 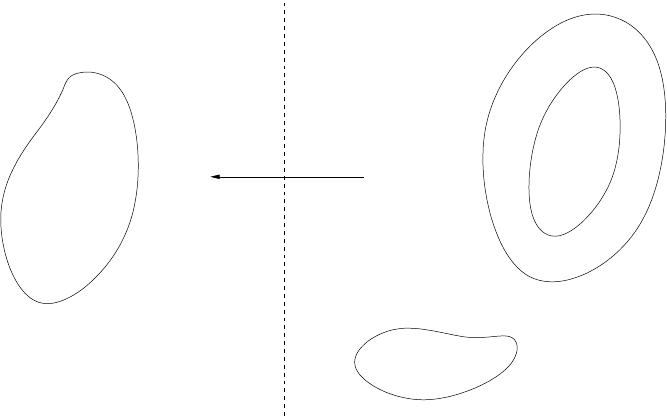
 \caption{The disk $D$ and its preimages: the annulus $A$ and the disk $D'$. The region $U_1$ maps homeomorphically onto $S^2 \setminus D$ whereas the region $U_2$, which contains both fixed critical points, maps by a degree $2$ covering to $S^2 \setminus D$. In this and other figures, $\times$ denotes a critical value, $\ast$ a critical point and $\circledast$ a fixed critical point.}
 \label{f:quadlike}
 \end{figure} 
 
 The main theorem in this section is the following. Let $F$ be a cubic Thurston map with two fixed critical points.
 
 \begin{thm}\label{t:quadlike}
  Suppose $\Gamma$ is an irreducible obstruction of $F$ such that $\Sigma \setminus \Gamma$ has a disk component $D$ containing precisely two critical values, neither of which is a fixed critical point of $F$, and that $F^{-1}(D)$ contains a non-disk component. Then every curve in $\Gamma$ will have precisely three preimages. Moreover, $\Gamma$ must contain a Levy cycle.
 \end{thm}
 
 We now prove some preliminary results to allow us to prove this theorem.
 
 \begin{lem}\label{l:sep}
 Let $\Gamma$ be an irreducible obstruction of $F$ such that $\Sigma \setminus \Gamma$ has a disk component $D$ containing precisely two critical values, neither of which is a fixed critical point of $F$, and that $F^{-1}(D)$ contains a non-disk component. Then $\Gamma$ does not contain any curves which separate the two fixed critical points.
 \end{lem}

 \begin{proof}  
  We will call the critical values which are not the fixed critical points \emph{free} critical values. Since the two free critical values are contained in a disk component $D$ of $S^2 \setminus \Gamma$, no curve in $\Gamma$ can separate them. Define
  \[
   \widetilde{\Gamma} = \{ \gamma \in \Gamma \, \mid \, \gamma \text{ separates the two fixed critical points of } F \}.
  \]
Observe that $\Gamma \setminus \widetilde{\Gamma}$ is non-empty since $\partial D$ does not separate the two fixed critical points. Now suppose $\gamma \in \Gamma \setminus \widetilde{\Gamma}$. There are two possible cases, see Figure~\ref{f:newtonproof}. 
\begin{itemize}
 \item[(i)] {Each component of $\Sigma \setminus \gamma$ contains two critical values. Furthermore, one component must contain both fixed critical points.}
 \item[(ii)] {One component of $\Sigma \setminus \gamma$ contains no critical values.}
\end{itemize}

\begin{figure}[ht]
 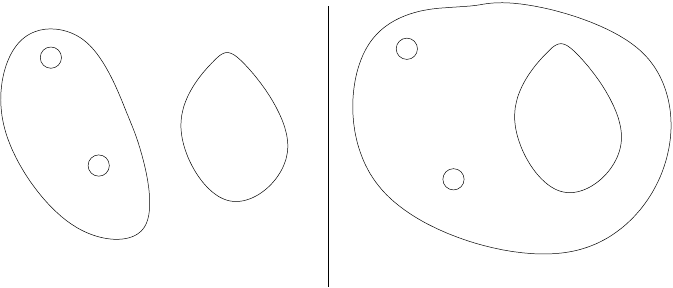
 \caption{The two possible cases for the curve $\gamma$ in Lemma~\ref{l:sep}. One the left, each component of $\Sigma \setminus \gamma$ contains two critical values. On the right, one component of $\Sigma \setminus \gamma$ contains no critical values.}
 \label{f:newtonproof}
 \end{figure} 

 In the first case, denote the component of $\Sigma \setminus \gamma$ containing the two fixed critical points by $B$. By an application of Lemma~\ref{l:RH}, we see that $F^{-1}(B)$ consists of an annulus $A$ and a disk $B'$. The disk is a subset of the region $U_1$ and the annulus is contained in the region $U_2$; furthermore, the annulus contains both of the fixed critical points of $F$. The components of $F^{-1}(\gamma)$ are the curve $\partial B'$ and the two curves of $\partial A$ (see Figure~\ref{f:newtonproof1}). It follows that none of the components of $F^{-1}(\gamma)$ separate the two fixed critical points -- in other words, none of the preimage curves is in $\widetilde{\Gamma}$. 
 
 \begin{figure}[ht]
 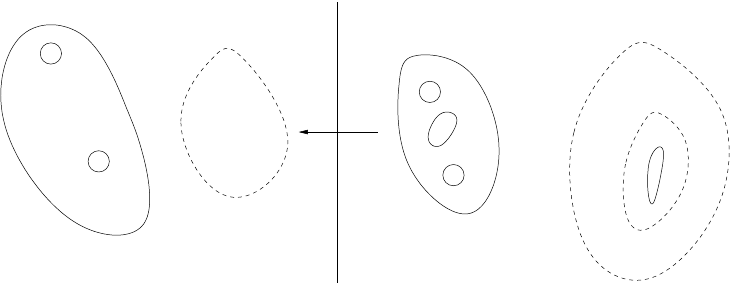
 \caption{The first case in Lemma~\ref{l:sep}. The curve $\gamma$ and its preimages are unbroken lines, and the curve $\partial D$ and some of its preimages are broken lines. Note that none of the preimages of $\gamma$ separates the two fixed critical points.}
 \label{f:newtonproof1}
 \end{figure} 
 
 In the second case, denote the component of $\Sigma \setminus \gamma$ which contains no critical values by $C$. Since $C$ contains no critical values, the set $F^{-1}(C)$ consists of three disks, none of which contains any critical points. It follows that none of the preimages of $\gamma$ can separate the two fixed critical points (see Figure~\ref{f:newtonproof2}), and so no curve in $F^{-1}(\gamma)$ can belong to $\widetilde{\Gamma}$.
 
 \begin{figure}[ht]
 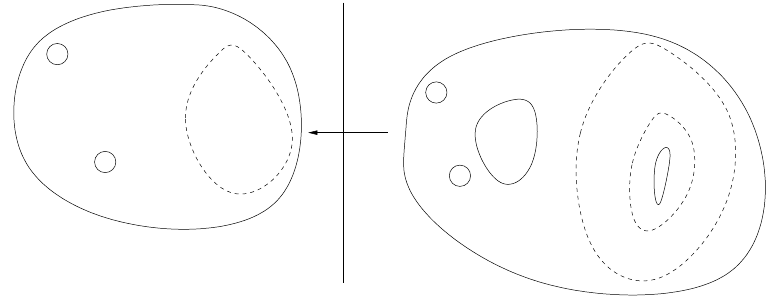
 \caption{The second case in Lemma~\ref{l:sep}. Again, the curve $\gamma$ and its preimages are unbroken lines, and the curve $\partial D$ and some of its preimages are broken lines. As with the first case, none of the preimages of $\gamma$ separate the two fixed critical points.}
 \label{f:newtonproof2}
 \end{figure} 
 
 It follows that any irreducible obstruction containing a curve in $\Gamma \setminus \widetilde{\Gamma}$ cannot also contain a curve in $\widetilde{\Gamma}$. Since by assumption $\partial D \in \Gamma$ and $\partial D \in \Gamma \setminus \widetilde{\Gamma}$, we see that our irreducible obstruction cannot contain any curves which separate the two fixed critical points.
 \end{proof}

We are now ready to prove Theorem~\ref{t:quadlike}.
 
 \begin{proof}[Proof of Theorem~\ref{t:quadlike}]
  By Lemma~\ref{l:sep}, no curve in the irreducible obstruction $\Gamma$ can separate the two fixed critical points. It follows that if $\gamma \in \Gamma$, then (using Lemma~\ref{l:RH} part (a)) the set $F^{-1}(\gamma)$ will consist of three components, and so each preimage will map homeomorphically onto $\gamma$. Since $\Gamma$ contains only finitely many curves, any periodic cycle (under $F^{-1}$) in $\Gamma$ must be a Levy cycle.
 \end{proof}
 
 \begin{rem}
  It is perfectly possible that the Levy cycle found above is degenerate but not removable. An example is given in the appendix.
 \end{rem}

 \subsection{Newton-like case}
 
 This case is very similar to that considered by Tan in \cite{Tan:Newton}. There it was shown that if a cubic Thurston map with three fixed critical points has an obstruction, then this obstruction contains a Levy cycle. The present paper generalises Tan's result to the case where there are only two fixed critical points. In this section, we will assume that we are not in the quadratic-like case. That is, we assume that there does not exist a disk component of $\Sigma \setminus \Gamma$ containing both free critical values and whose preimage contains a non-disk component.  The main result for the Newton-like case is the following.
 
 \begin{thm}\label{t:Newton}
 Let $F$ be a cubic Thurston map with two fixed critical points, and let $\Gamma$ be an irreducible obstruction of $F$ such that $\Sigma \setminus \Gamma$ contains a disk component $D$ that has a preimage which is a non-disk component. If furthermore $D$ contains a fixed critical point, then $\partial D \in \Gamma$ and $\{ \partial D \}$ is a Levy cycle.
\end{thm}

The proof of this theorem needs a little care. In some sense, the most important case (at least when it comes to considering matings) is when $D$ contains precisely two critical values of $F$. However, we first turn our attention to the case where $D$ contains more than two critical values. In this case, it follows from Lemma~\ref{l:RH} and Proposition~\ref{p:isotopiccontain} that the disk component $D$ has exactly one preimage and that this preimage is isotopically contained in $D$ relative to $P_F$ (see Proposition~\ref{p:pDisLevy}). The following lemma is the key observation.

\begin{lem}\label{l:UinD1}
 Let $\Gamma$ be an irreducible obstruction for a postcritically finite branched covering $F$. Suppose $U$ is a preimage component of a disk component $D$ in $\Sigma \setminus \Gamma$ and $U$ is isotopically contained in $D$ relative to $P_F$. Then if an element of $\partial U$ is an element of $\Gamma$, it is isotopic to $\partial D$.
\end{lem}

\begin{proof}
 Let $\gamma' \in \partial U$ be an element of $\Gamma$. Since $U \cap P_F \subseteq D \cap P_F$, we see that $\gamma' \subseteq \overline{D}$ (isotopically). But this means that if $\gamma'$ is not isotopic to $\gamma$, then $D$ is not a disk component of $\Sigma \setminus \Gamma$, which is a contradiction. 
\end{proof}

In the case where $D$ contains more than two critical values, the following immediate corollary will be useful.

\begin{cor}\label{c:UinD}
 Let $\Gamma$ be an irreducible obstruction for a Thurston map $F$. Suppose $U$ is the only preimage component of a disk component $D$ and that $U$ is isotopically contained in $D$ relative to $P_F$. Then there exists a curve in $\partial U$ which is isotopic to $\partial D$.
\end{cor}

\begin{proof}
 First observe that since $\Gamma$ is an irreducible obstruction, the curve $\gamma = \partial D$ in $\Gamma$ must have a preimage (up to isotopy) in $\Gamma$. Let $\gamma'$ be such a curve. By Lemma~\ref{l:UinD1}, $\gamma'$ is isotopic to $\gamma$. Thus there exists a preimage of $\gamma$ which is isotopic to $\gamma$.
\end{proof}

We now may take care of the case where $D$ contains more than two critical values.

\begin{prop}\label{p:pDisLevy}
 Let $F$ be a cubic Thurston map with two fixed critical points, and let $\Gamma$ be an irreducible obstruction of $F$ such that $\Sigma \setminus \Gamma$ contains a disk component $D$ which contains at least three critical values. Then $\{ \partial D \}$ is a Levy cycle.  
 \end{prop}
 
 \begin{proof}
 We claim that if $D$ contains three or four critical values, then it has exactly one preimage component $U$ and $U$ is isotopically contained in $D$ relative to $P_F$. By considering complements in Lemma~\ref{l:RH}, we see that if $D$ contains three or four critical values, then $F^{-1}(D)$ consists of just one component, $U$. Furthermore, since $D$ contains more than two critical values, at least one of them must be a fixed critical point $c$. Since $c \in D$ and $c \in U$, we see by Proposition~\ref{p:isotopiccontain} that $U$ is isotopically contained in $D$ relative to $P_F$. It then follows from Corollary~\ref{c:UinD} that there exists a curve in $F^{-1}(\partial D)$ which is isotopic to $\partial D$. We now split into cases. Denote $\gamma = \partial D$.
 
  If $D$ contains exactly three critical values, then by Lemma~\ref{l:RH} we see that $\gamma$ has two preimages, $\gamma'$ and $\gamma''$ which map by degree $1$ and degree $2$ respectively onto $\gamma$. Note that in particular, $\gamma''$ bounds the component of $\Sigma \setminus U$ which contains the unique critical point of $F$ which is not contained in $D$ (see Figure~\ref{f:newton2}).
  
  \begin{figure}[ht]
  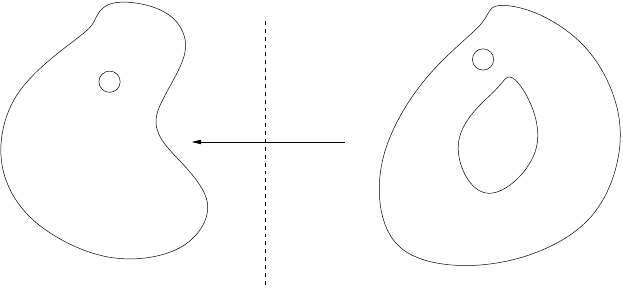
  \caption{The case where $D$ contains three critical values and so has one preimage component $U$. Here the curve $\gamma = \partial D$ has two preimages, $\gamma'$ and $\gamma''$ which map onto $\gamma$ by degree $1$ and $2$ respectively.}
  \label{f:newton2}
  \end{figure}
  
  If $\gamma'$ is not in $\Gamma$, then $\gamma''$ must be isotopic to $\gamma$. But then $\Gamma' = \{ \gamma \}$ is an irreducible submulticurve of $\Gamma$,  with $\lambda(\Gamma') = 1/2$. It follows that $\Gamma$ is not an irreducible obstruction, which is a contradiction. Thus $\gamma'$ is in $\Gamma$ and so must be isotopic to $\gamma$ by Lemma~\ref{l:UinD1}. Since $\gamma'$ maps homeomorphically onto $\gamma$, we then see that $\{ \gamma \}$ is a Levy cycle. 
  
  If $D$ contains exactly four critical values, then by Lemma~\ref{l:RH}, then $\gamma$ has three preimages, each of which maps homeomorphically onto $\gamma$. At least one of these is isotopic to $\gamma$ by Corollary~\ref{c:UinD}, and so again $\{ \gamma \}$ is a Levy cycle.
 \end{proof}

 We now turn our attention to the most important case (at least with regards to the sequel to this paper), where the disk $D$ contains precisely two critical values, at least one of which is a fixed critical point $c$.

 \begin{lem}\label{l:partialA}
  Let $F$ be a cubic Thurston map with two fixed critical points, and let $\Gamma$ be an irreducible obstruction of $F$. Suppose there is a disk component $D$ of $\Sigma \setminus \Gamma$ such that
  \begin{itemize}
   \item $F^{-1}(D)$ contains a non-disk component and
   \item $D$ contains precisely two critical values, at least one of which is a fixed critical point $c$ of $F$.
  \end{itemize}
Then $F^{-1}(D)$ contains an annulus $A$ which is isotopically contained in $D$ relative to $P_F$, and one of the curves in $\partial A$ is in $\Gamma$. 
 \end{lem}
 
 \begin{proof}
  By Lemma~\ref{l:RH}, the components of $F^{-1}(D)$ are an annulus $A$, and a disk $\Delta$ which contains no critical points. Furthermore, since $c \in A$, it follows from Proposition~\ref{p:isotopiccontain} that $A$ is isotopically contained in $D$ relative to $P_F$ (see Figure~\ref{f:newton}). To obtain a contradiction, suppose neither curve in $\partial A$ is in $\Gamma$. Then $\partial \Delta$ must be in $\Gamma$. 
  
  \begin{figure}[ht]
  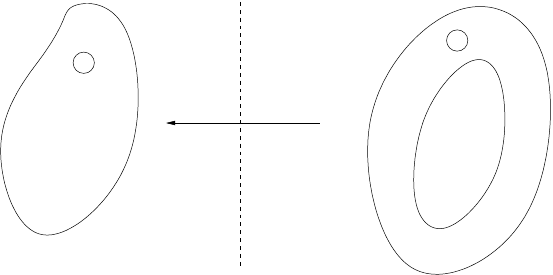
  \caption{When the disk contains two critical values, one of which is fixed, then there exists an annular component $A$ in $F^{-1}(D)$ which is isotopically contained in $D$ relative to $P_F$.}
  \label{f:newton}
  \end{figure}
  
  Firstly, suppose that there exists a disk component $\Delta'$ of $F^{-n}(D)$ such that $F^{-1}(\Delta')$ contains a non-disk component. Then by Lemma~\ref{l:RH}, $\Delta'$ must contain two critical values. Since $\Delta'$ does not contain a fixed critical point (for otherwise, by taking forward iterates, we would see that $\Delta$ would also contain a fixed critical point, which is a contradiction) we see that $\Delta'$ is disjoint from $D$ and that $\Delta'$ contains the two free critical values. Since $F^{-1}(\Delta')$ contains a non-disk component, we see that we are in the quadratic-like case, contradicting our assumption at the start of the section. Thus all components of $F^{-n}(D)$ must be disks. But then by Proposition~\ref{p:critvals}, $\Gamma$ contains a removable Levy cycle $L$. However, $\partial D \notin L$ (since $D$ contains critical points), which means that $\Gamma$ is not irreducible. This is a contradiction. Hence there must exist a curve in $\partial A$ which is in $\Gamma$. 
 \end{proof}

Combining Lemma~\ref{l:partialA} with Lemma~\ref{l:UinD1} gives the following.
 
 \begin{prop}\label{p:2critvcase}
  Let $F$ be a cubic Thurston map with two fixed critical points, and let $\Gamma$ be an irreducible obstruction of $F$. Suppose there is a disk component $D$ of $\Sigma \setminus \Gamma$ such that
  \begin{itemize}
   \item $F^{-1}(D)$ contains a non-disk component and
   \item $D$ contains precisely two critical values, at least one of which is a fixed critical point $c$ of $F$.
  \end{itemize}
Then $\{ \partial D \}$ is a Levy cycle.
 \end{prop}

 \begin{proof}
  By Lemma~\ref{l:partialA}, $F^{-1}(D)$ contains an annulus which is isotopically contained in $D$ relative to $P_F$. Therefore, by Lemma~\ref{l:UinD1}, we see that if an element $\gamma'$ of $\partial A$ is in $\Gamma$, then it is isotopic to $ \gamma = \partial D$. Since $F^{-1}(\gamma)$ consists of three components, we see that $\gamma'$ maps homeomorphically onto $\gamma$ and so $\{ \gamma \}$ is a Levy cycle.  
  \end{proof}

  The proof of Theorem~\ref{t:Newton} now readily follows.
 
 \begin{proof}[Proof of Theorem~\ref{t:Newton}.]
  Since $D$ is a disk component of $\Sigma \setminus \Gamma$ such that $F^{-1}(D)$ contains a non-disk component, it must contain at least two critical values, at least one of which (by assumption) is a fixed critical point. If $D$ contains more than two critical values, then by Proposition~\ref{p:pDisLevy}, the multicurve $\{ \partial D \}$ is a Levy cycle. If $D$ contains exactly two critical values, then by Proposition~\ref{p:2critvcase}, either $\{ \partial D \}$ is a Levy cycle or $\Gamma$ contains a degenerate Levy cycle.
 \end{proof}

 \subsection{Proof of Main Theorem \ref{mthm}}
 
 The proof of Main Theorem~\ref{mthm} is now immediate.
 
 \begin{proof}[Proof of Main Theorem~\ref{mthm}]
  Let $\Gamma$ be an irreducible obstruction and suppose $\Gamma$ is not a removable Levy cycle. By Proposition~\ref{p:notremov}, we must therefore be in the Newton-like case or the quadratic-like case. If we are in the Newton-like case, Theorem~\ref{t:Newton} asserts that $\Gamma$ must be a Levy cycle containing just one curve. On the other hand, if we are in the quadratic-like case, then it follows that Theorem~\ref{t:quadlike} that $\Gamma$ must be a Levy cycle.
 \end{proof}

 \subsection{Matings and the Proof of Main Theorem \ref{mthm2}}
 
 In this section, we outline the general theory of matings and give a proof of Main Theorem \ref{mthm2}. In actual fact, most of the work has already been carried out in Main Theorem~\ref{mthm}. The reader interested in finding out more details should refer to \cite{Milnor:mating,NotionsofMating,ShishTan,Tan:quadmat}. Let $f$ and $g$ be monic degree $d$ polynomials. We define
\[
 \widetilde{\C} = \C \cup \{ \infty \cdot e^{2 \pi i t} : t \in \R / \Z \},
\]
the complex plane compactified with the circle of directions at infinity. We then continuously extend the two polynomials to the circle at infinity by defining
\[
  f(\infty \cdot e^{2 \pi i t}) =  \infty \cdot e^{2 d \pi i t} \quad \text{and} \quad g(\infty \cdot e^{2 \pi i t}) =  \infty \cdot e^{2 d \pi i t}.
\]
Label this extended dynamical plane of $f$ (respectively $g$) by $\widetilde{\C}_f$ (respectively $\widetilde{\C}_g$). We create a topological $2$-sphere $\Sigma_{f,g}$ by gluing the two extended planes together along the circle at infinity:
\[
 \Sigma_{f,g} = \widetilde{\C}_f \uplus \widetilde{\C}_g / \sim
\]
where $\sim$ is the relation which identifies the point $\infty \cdot e^{2 \pi i t} \in \widetilde{\C}_f$ with the point $\infty \cdot e^{- 2 \pi i t} \in \widetilde{\C}_g$. The \emph{formal mating} is then defined to be the branched covering $f \uplus g \colon \Sigma_{f,g} \to \Sigma_{f,g}$ such that
\begin{align*}
    f \uplus g|_{\widetilde{\C}_f} \, =& \, f \quad \textrm{and} \\
    f \uplus g|_{\widetilde{\C}_g} \, =& \, g.
\end{align*}  
 
 \begin{proof}[Proof of Main Theorem~\ref{mthm2}].
 Observe that if $f$ and $g$ are monic postcritically finite polynomials each possessing a fixed critical point, then the formal mating $f \uplus g \colon \Sigma_{f,g} \to \Sigma_{f,g}$ is a postcritically finite branched covering of a topological $2$-sphere which has (at least) two fixed critical points. By Main Theorem~\ref{mthm}, any obstruction to this mating must contain a Levy cycle.  
 \end{proof}
 
 \appendix
 
 \section{Examples of Obstructions}
 
 We conclude this article by giving examples of the types of obstructions given in this article. All of these examples will be constructed as (formal) matings $f \uplus g$ of pairs of polynomials in $\mathcal{S}_1$, the space of cubic polynomials with a fixed marked critical point. The polynomials in $\mathcal{S}_1$ may be put in the normal form
 \[
  f_a(z) = z^3 - 3a^2z+ 2a^3 + a
 \]
 with the fixed critical point at $a$ and the free critical point at ${-}a$. Thus one may identify the polynomial $f_a$ with the point $a \in \C$. For more details of this space (and other slices of the space of cubic polynomials) the reader is referred to \cite{CP1,CP3}. 
 
 Before we begin, we describe the polynomials we will use to construct the examples. Their positions in $\mathcal{S}_1$ are marked in Figure~\ref{f:S1arrows}. (Filled) Julia sets of these maps, with important rays marked, are in Figure~\ref{f:Juliasets}.
 
 \begin{itemize}
  \item Let $f_0$ be the polynomial which is the landing point of the parameter ray of angle $0$. For $f_0$, we have that the free critical point $-a$  maps onto a repelling fixed point $v$ after one iterate. The corresponding parameter value is $a_0 = 0.5$. 
  \item Let $f_1$ be the centre of the hyperbolic component whose root is the landing point of the rays of angles $1/24$ and $2/24$. In this case the free critical point $-a$ belongs to a period $2$ cycle. The corresponding parameter is $a_1 = 0.500869\dotsc + 0.259677\dotsc i$.
  \item Let $f_2 = \overline{f_1}$ be the complex conjugate of $f_1$ (so that $f_2(z) = \overline{f_1(z)}$ for all $z \in \C$). 
  \item Let $f_3 = -f_2 = - \overline{f_1}$.
 \end{itemize}

 \begin{figure}[ht]
  \centering
  \includegraphics[width=0.6\textwidth]{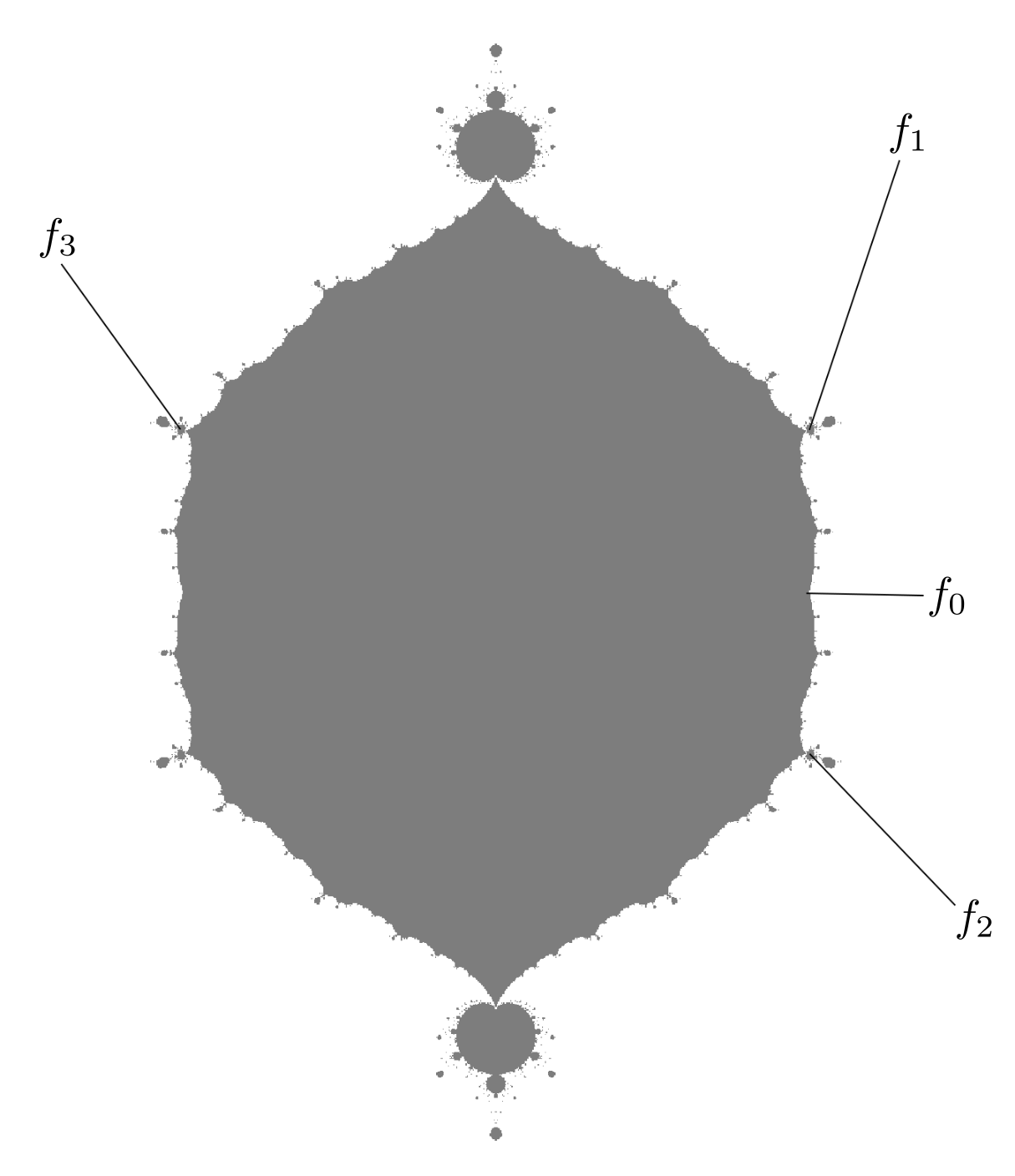}
  \caption{The parameter space $\mathcal{S}_1$ with the positions of $f_i$, $i=0,\dotsc3$ marked. \label{f:S1arrows}}
 \end{figure}

 \begin{figure}
  \centering
  \begin{subfigure}[h]{0.4\textwidth}
   \includegraphics[width=\textwidth]{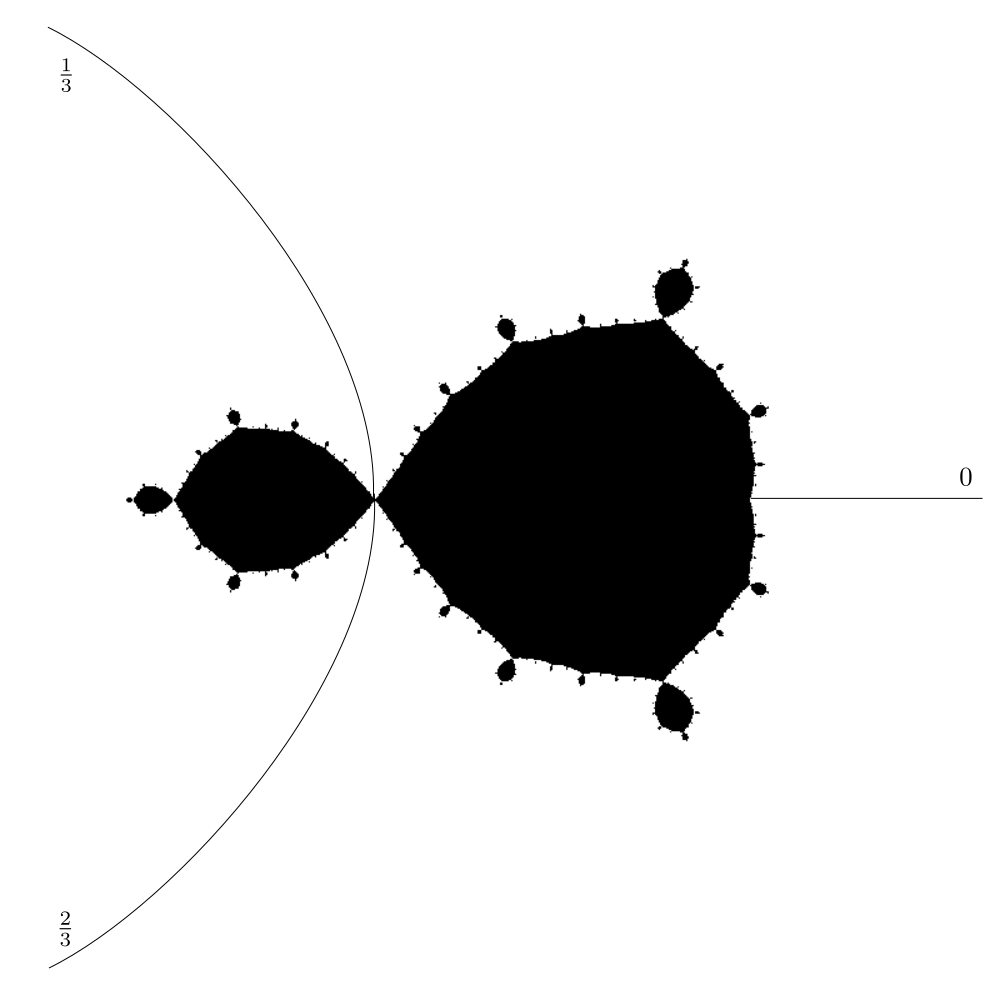}
   \caption{The Julia set for $f_0$}
  \end{subfigure}
\quad
\begin{subfigure}[h]{0.4\textwidth}
   \includegraphics[width=\textwidth]{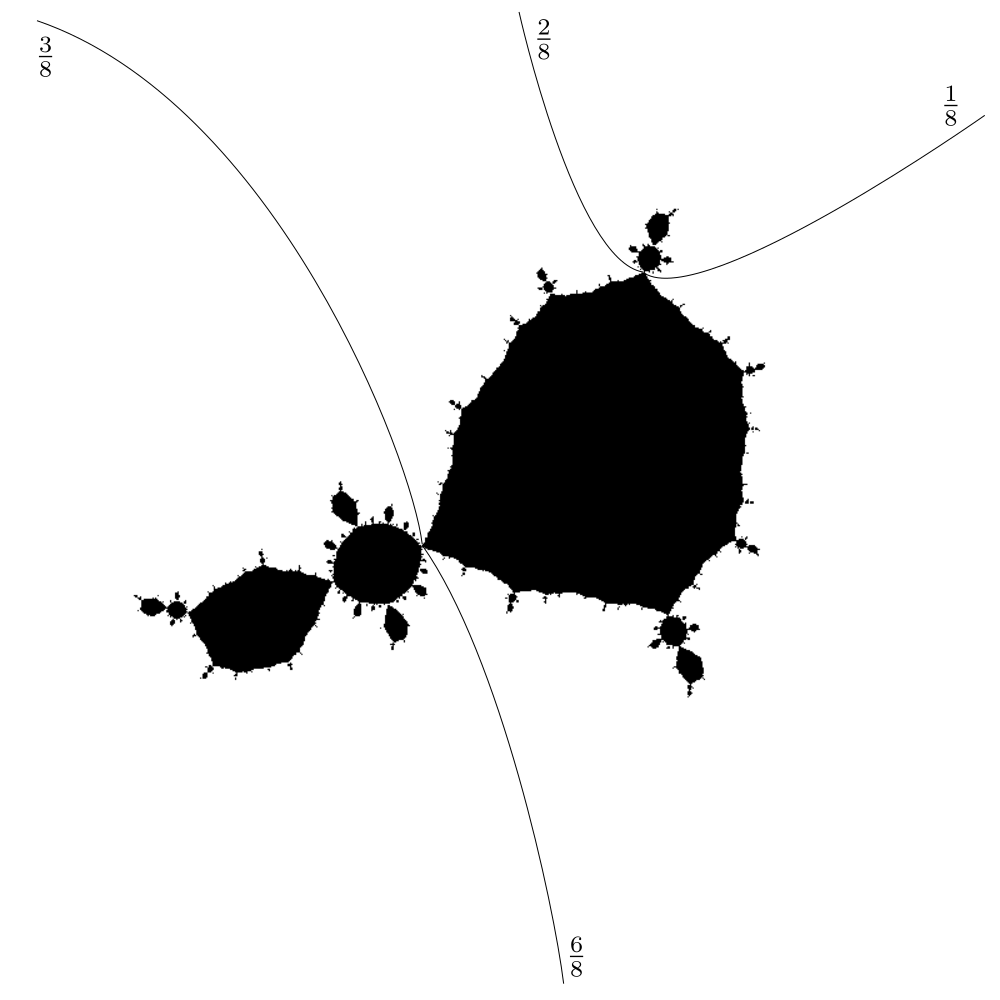}
   \caption{The Julia set for $f_1$}
  \end{subfigure}
  \\
  \begin{subfigure}[h]{0.4\textwidth}
   \includegraphics[width=\textwidth]{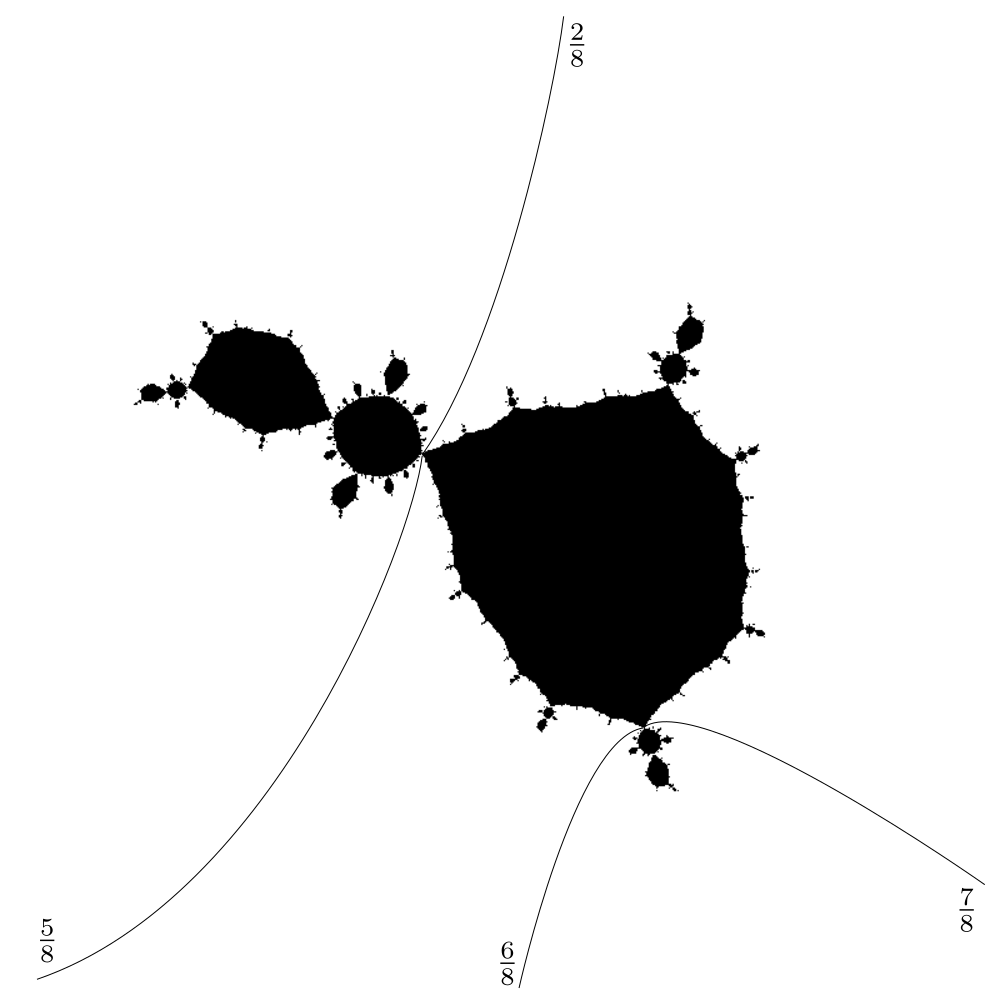}
   \caption{The Julia set for $f_2$}
  \end{subfigure}
\quad
\begin{subfigure}[h]{0.4\textwidth}
   \includegraphics[width=\textwidth]{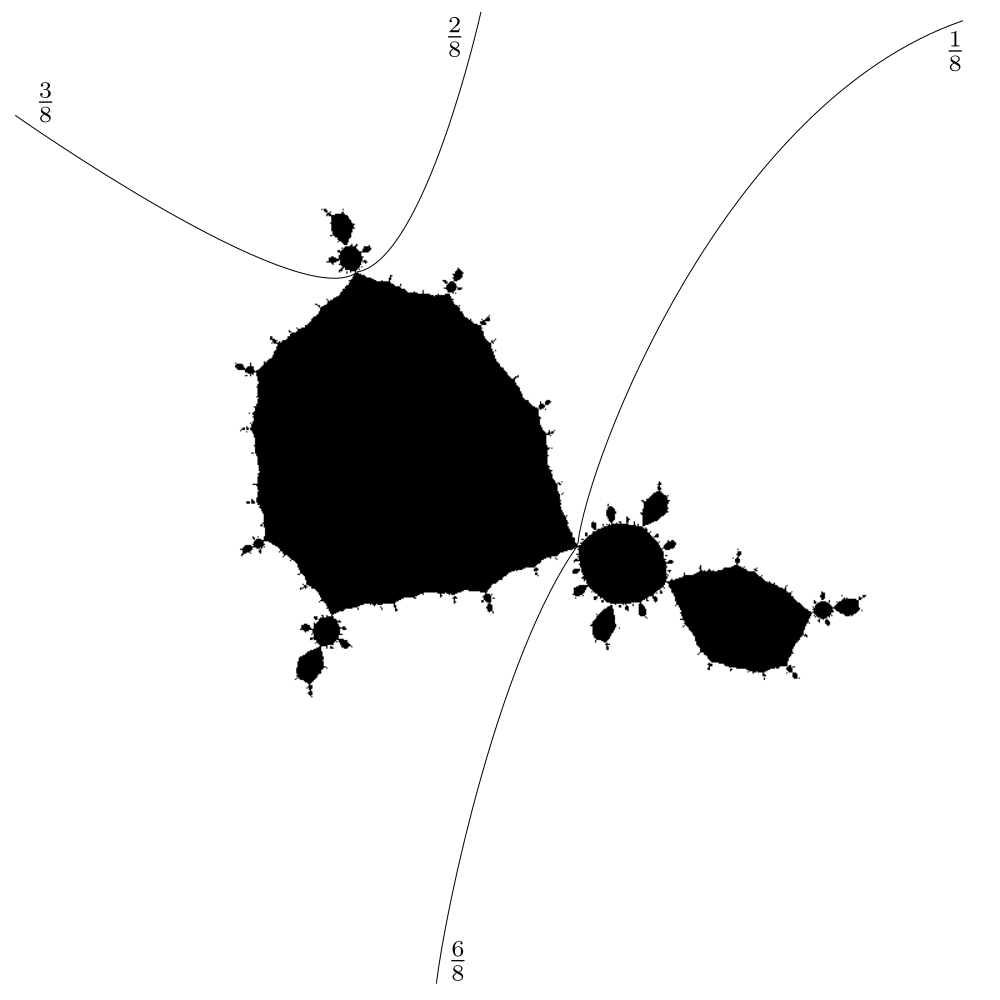}
   \caption{The Julia set for $f_3$}
  \end{subfigure}
  \caption{Filled Julia sets for the polynomials $f_i$, $i=0,\dotsc3$.}
  \label{f:Juliasets}
 \end{figure}

 We will denote the dynamical ray of angle $t$ for the map $f$ (respectively $g$) by $R_f(t)$ (respectively $R_g(t)$). In the formal mating $F = f \uplus g$, we also denote by $R_F(t)$ the set $\overline{R_f(t) \cup R_g(-t)}$.
 
 \subsection{Quadratic-like case}
 
 In this section there are two examples. The second example is noteworthy since it constructs a degenerate Levy cycle which is not removable.
 
 \begin{ex}
  Let $f = f_1$ and $g=f_2$ and $F = f \cup g$. Let
  \[
   \gamma_1 = R_F(1/8) \cup R_F(2/8) \qquad \text{and} \qquad \gamma_2 = R_F(3/8) \cup R_F(6/8).
  \]
It is easy to see that $F$ maps $\gamma_1$ homeomorphically onto $\gamma_2$ and $\gamma_2$ homeomorphically onto $\gamma_1$. Furthermore, $\gamma_1$ bounds a disk component which contains two critical values, neither of which is fixed. Hence, we are in the quadratic-like case. 
 \end{ex}

 \begin{ex}
  Consider the formal mating $F = f \uplus g$ where $f = g = f_0$. Then if $\gamma$ is a tubular neighbourhood of $R_F(0)$, then $\Gamma = \{ \gamma \}$ is a degenerate Levy cycle. Furthermore, $\gamma$ bounds a disk region $D$ which contains the two critical values of $F$ which are not fixed critical points, and so we are in the quadratic-like case. Since the preimage of $D$ consists of two components, one isotopic to $D$ and the other an annulus, we see that $\Gamma$ is a degenerate Levy cycle but not a removable Levy cycle.
 \end{ex}

 \subsection{Newton-like Case}
 
 \begin{ex}
  Consider the formal mating $F = f_1 \uplus f_3$. In this case, observe that
  \[
   \gamma = R_F(1/8) \cup R_F(2/8) \cup R_F(3/8) \cup R_F(6/8)
  \]
forms a closed curve and $F \colon \gamma \to \gamma$ is a homeomorphism. Since $\gamma$ separates the two fixed critical points (and so each disk component of the complement of $\gamma$ contains a fixed critical point), we see that the Levy cycle $\Gamma = \{ \gamma \}$ belongs to the Newton-like case.
 \end{ex}

\bibliographystyle{alpha}
\bibliography{cubicobs}

\begin{thebibliography}{{Tan}97}

\bibitem[AR16]{AspRoe}
M.~Aspenberg and P.~Roesch.
\newblock Newton maps as matings of cubic polynomials.
\newblock {\em Proc. Lond. Math. Soc. (3)}, 113(1):77--112, 2016.

\bibitem[BFH92]{BFH}
B.~Bielefield, Y.~Fisher, and J.~Hubbard.
\newblock The classification of critically preperiodic polynomials as dynamical
  systems.
\newblock {\em Journal AMS}, 5:721--762, 1992.

\bibitem[BM17]{BonkMeyer}
M.~Bonk and D.~Meyer.
\newblock {\em Expanding {T}hurston maps}, volume 225 of {\em Mathematical
  Surveys and Monographs}.
\newblock American Mathematical Society, Providence, RI, 2017.

\bibitem[BM18]{CP3}
A.~Bonifant and J.~W. Milnor.
\newblock Cubic polynomials maps with periodic critical point, {Part III}:
  {External} rays.
\newblock Manuscript, in preparation, 2018.

\bibitem[DH93]{DouadyHubbard:Thurston}
A.~Douady and J.~H. Hubbard.
\newblock A proof of {Thurston's} topological characterization of rational
  functions.
\newblock {\em Acta. Math.}, 171:263--297, 1993.

\bibitem[Gan59]{Gantmacher}
F.~R. Gantmacher.
\newblock {\em {The Theory of Matrices}}.
\newblock Chelsea, 1959.

\bibitem[Mil04]{Milnor:mating}
J.~W. Milnor.
\newblock Pasting together {J}ulia sets: a worked out example of mating.
\newblock {\em Experiment. Math.}, 13:55--92, 2004.

\bibitem[Mil06]{MilnorComplex}
J.~W. Milnor.
\newblock {\em Dynamics in One Complex Variable}.
\newblock Princeton University Press, 3rd edition, 2006.

\bibitem[Mil09]{CP1}
J.~W. Milnor.
\newblock Cubic polynomial maps with periodic critical orbit. {I}.
\newblock In {\em Complex dynamics}, pages 333--411. A K Peters, Wellesley, MA,
  2009.

\bibitem[MP12]{NotionsofMating}
D.~Meyer and C.~L. Petersen.
\newblock On the notions of mating.
\newblock {\em Ann. Fac. Sci. Toulouse Math. (6)}, 21(5):839--876, 2012.

\bibitem[Sha20]{CubicmatingsPt2}
T.~Sharland.
\newblock Matings of cubic polynomials with a fixed critical point, {Part II}:
  Topological obstructions.
\newblock In preparation, 2020.

\bibitem[ST00]{ShishTan}
M.~Shishikura and {Tan~L.}
\newblock A family of cubic rational maps and matings of cubic polynomials.
\newblock {\em Experimental Math.}, 9:29--53, 2000.

\bibitem[{Tan}92]{Tan:quadmat}
{Tan~L.}
\newblock Matings of quadratic polynomials.
\newblock {\em Ergodic Th. Dyn. Sys.}, 12:589--620, 1992.

\bibitem[{Tan}97]{Tan:Newton}
{Tan~L.}
\newblock Branched coverings and cubic {Newton} maps.
\newblock {\em Fund. Math.}, 154:207--260, 1997.

\end{thebibliography}
 
\end{document}